\pdfpkmode={lexmarkr}
\pdfmapfile{}

\documentclass[a4paper,reqno,12pt]{amsart}

\usepackage[english]{babel}
\usepackage[latin1]{inputenc}
\usepackage[T1]{fontenc}
\usepackage{graphicx}
\usepackage{placeins}
\usepackage[usenames]{xcolor}
\usepackage{url}
\usepackage{mathtools}
\usepackage{amssymb}
\usepackage[shortlabels]{enumitem}
\usepackage{csquotes}
\usepackage[a4paper,margin=1.4in]{geometry}
\usepackage[font=small,width=0.7\textwidth]{caption}

\newtheorem{theorem}{Theorem}[section]
\newtheorem{proposition}[theorem]{Proposition}
\newtheorem{lemma}[theorem]{Lemma}
\newtheorem{corollary}[theorem]{Corollary}
\newtheorem{remark}[theorem]{Remark}

\newcommand{\R}{\mathbb{R}}

\newcommand{\Ab}{{\bf A}}

\DeclareMathOperator{\supp}{supp}
\DeclareMathOperator{\Div}{div}
\DeclareMathOperator{\curl}{curl}
\DeclareMathOperator{\Osc}{Osc}

\begin{document}

\title[On the ground state energy of the Pauli operator]{On the 
semi-classical analysis of the groundstate energy of the Dirichlet Pauli operator 
in non-simply connected domains}

\author{Bernard Helffer}
\author{Mikael Persson Sundqvist}
\address[Bernard Helffer]{ Laboratoire de Math\'ematiques Jean Leray, 
Universit\'e de Nantes,  2 rue de la Houssini\`ere, 44322 Nantes, France
and Laboratoire de
Math\'{e}matiques, Universit\'{e} Paris-Sud, France. 
}
\email{bernard.helffer@univ-nantes.fr}
\address[Mikael Persson Sundqvist]{Lund University, Department of Mathematical 
Sciences, Box 118, 221\ 00 Lund, Sweden.}
\email{mickep@maths.lth.se}

\begin{abstract}
We consider the Dirichlet Pauli operator in bounded connected 
domains in the plane, 
with a semi-classical parameter. We show, in particular, that the
ground state energy of this Pauli operator will be exponentially small as the
semi-classical parameter tends to zero and estimate this decay rate.
This extends our results~\cite{HPS}, discussing the results of a recent 
paper by Ekholm--Kova\v{r}\'ik--Portmann~\cite{EKP}, to include also non-simply 
connected domains.
\end{abstract}
\maketitle
\section{Introduction}

Let $\Omega$ be a connected, regular domain in $\mathbb R^2$
and let $B=B(x)$ be a magnetic field  in $C^\infty(\bar \Omega)$  and $h>0$ a semiclassical parameter.
We are interested in the analysis of the ground state energy 
$ \lambda_{P_-}^D(h,\Ab, B,\Omega)$ of the Dirichlet realization of the Pauli operator
\[
P_-:= (hD_{x_1} -A_1)^2 + (hD_{x_2}-A_2)^2 - h B(x)\,.
\]
Here $D_{x_j} = -i \partial _{x_j}$ for $j=1,2\,$ and the vector potential
$\Ab=(A_1,A_2)$ satisfies
\begin{equation}\label{eq:0}
B(x) = \partial_{x_1} A_2 - \partial_{x_2} A_1\,.
\end{equation}
The reference to $\Ab$ is not necessary when $\Omega$ is simply connected (in this case, we omit it as in \cite{HPS})  but could play an important role if the domain is not simply connected.
It is a well-known fact that the Pauli operator is non-negative (this follows
by an integration by parts). This implies that
\[
\lambda_{P_-}^D(h,\Ab, B,\Omega) \geq 0\,.
\]
Under the assumption  that 
\begin{equation}\label{eq:signcondition}
\{ x \in \Omega : B(x) >0\} \neq \varnothing\,,
\end{equation}
we know from~\cite{EKP,HPS} that $\lambda_{P_-}^D(h,\Ab, B,\Omega)$ is exponentially 
small as the semi-classical parameter $h>0$ tends to $0$. In particular, 
T.~Ekholm, H.~Kova\v{r}\'ik and F.~Portmann \cite{EKP}  give a lower bound which has a universal character and will be extended
 in the following way
\begin{theorem}\label{thmEKPNSC}
 Let $\Omega$ be regular, bounded, connected in $\mathbb R^2$.  If $B$ does not vanish identically in $\Omega$ there exists $\epsilon >0$ such 
that, for all $ h >0$ and for all $\Ab$ such that $\curl \Ab=B$,
\begin{equation}\label{ineq1}
 \lambda_{P_-}^D(h,\Ab, B,\Omega) \geq   \lambda^D( \Omega) \, h^2\,  \exp \bigl(- \epsilon/h\bigr) \,.
 \end{equation}
 where 
$\lambda^D(\Omega)$ denotes the ground state energy of the Laplacian
on $\Omega$.
 \end{theorem}
In \cite{EKP}, the theorem is proven under the assumption that $\Omega$ is 
simply connected.  Without this assumption, their proof corresponds to a 
specific choice of the magnetic vector potential $\Ab$ associated with~$B$ in 
the definition of the Pauli operator. 
To be more precise, the magnetic potential is the restriction to $\Omega$ of a magnetic potential 
associated with a magnetic field $B$ which is given in a ball containing $\Omega$. 
Hence, in~\cite{EKP}, the circulation of the magnetic potential along the 
interior boundaries is determined by the flux of this extended~$B$ inside the 
corresponding hole. In general the circulations relative to each boundary of a hole are independent parameters. 
The extension to non-simply connected domains is relatively easy by domain monotonicity 
of the ground state energy in the case of the Dirichlet problem,
\[
\lambda_{P_-}^D(h,\Ab, B,\Omega) \geq  \lambda_{P_-}^D(h,\widetilde B,\widetilde \Omega) \,,
\]
as soon as we have constructed an  extension of  $\Ab$ and $B$ to $\widetilde \Omega$. We will actually proceed more directly.  Note  however that 
we might end up  in this way with something very far from the optimal $\epsilon$ in~\eqref{ineq1}, obtained in the simply connected case.

The proof in~\cite{EKP}  gives a way of computing some lower bound for 
$\epsilon$, by considering the oscillation of $\psi$ for any solution of 
$\Delta \psi =B$ and optimizing over $\psi$. We will just use a specific 
choice generalizing to the non-simply connected case the choice proposed 
in~\cite{HPS}. The main theorem in~\cite{HPS} was 
\begin{theorem} \label{th1.3}
 If $B(x) >0$, $\Omega$ is simply connected  and if $\psi_0$ is the solution of 
\[
\Delta \psi_0 = B(x) \text{ in } \Omega \,,\, \psi_{0/\partial \Omega} =0\,,
\]
then, for any $h>0$, 
\[
 \lambda_{P_-}^D(h, B,\Omega) \geq   \lambda^D(\Omega) \, h^2\,  \exp \bigl(2 \inf \psi_0 /h\bigr) \,.
\]
In the semi-classical limit,
\[
\lim_{h\rightarrow 0} h \log  \lambda_{P_-}^D(h,B,\Omega) \leq 2 \inf \psi_0\,.
\]
Here, $\lambda^D(\Omega)$ denotes the ground state energy of the Dirichlet
Laplacian in $\Omega$.
\end{theorem}
In the non-simply connected case, such formulation could be wrong. The result 
 could depend on the circulations of the magnetic potential along 
the different components of the boundary. The aim of this paper is  to analyze
this problem and to show that in the semi-classical limit the circulation effects disappear. Our main result is
\begin{theorem} \label{th1.3nsc}
 If $B(x) >0$,  $\Omega$ is connected, and if $\psi_0$ is the solution of 
\[
\Delta \psi_0 = B(x) \text{ in } \Omega \,,\, \psi_{0/\partial \Omega} =0\,,
\]
then, for any $\Ab$ such that $\curl \Ab = B$,  
\[
\lim_{h\rightarrow 0} h \log  \lambda_{P_-}^D(h,\Ab,B, \Omega) =  2 \inf \psi_0\,.
\]
\end{theorem}
The proof will use strongly the gauge invariance of the problem.

\begin{remark}
One can of course  think of getting an upper bound  by considering  a simply connected domain $\widehat \Omega\subset \Omega$ and use monotonicity arguments (with respect to the inclusion of domains). The question of an  optimal $\widehat \Omega$ is in this perspective relevant. Similar questions arise for the analysis of Aharonov--Bohm operators in~\cite{HHOHOO}. We will not continue in this direction which, in the semi-classical limit, does not lead to optimal results.
\end{remark}

The paper is organized as follows. In Section~\ref{s2}, we recall some basic results on magnetic potentials in the non-simply connected case.
In Section~\ref{s3}, we give the proof of Theorem~\ref{thmEKPNSC}, following the strategy of Ekholm--Kova\v{r}\'ik--Portmann. In Section~\ref{s4}, we discuss gauge invariance and a first application for improving Theorem~\ref{thmEKPNSC}. In Section~\ref{s7}, we present a very explicit form of the Hodge-De~Rham theory for twodimensional domains which is then used for the control of the oscillation of the generating function of the magnetic potential. We then implement the gauge invariance. Section~\ref{s6} is devoted to the upper bound showing the asymptotic optimality of the lower bound.  Finally, in Section~\ref{sannulus} the case of the annulus is analyzed in more details together with numerical computations.

\section{On magnetic potentials in not (necessarily) simply connected domains}\label{s2}
In this section, we  explore the question of existence and uniqueness for a magnetic vector potential with given circulations when the domain is not simply connected. 

\subsection{Canonical choice of the magnetic potential}\label{ss2.1}
 Following what is done for example in superconductivity (see~\cite{FH}), given some magnetic potential $\Ab$  in $\Omega$ satisfying \eqref{eq:0}, we can 
after a gauge transformation assume that $\Ab $ satisfies, in addition
to~\eqref{eq:0},
\begin{equation}\label{eq:5a}
 \, \Div \Ab = 0\text{ in } \Omega\,;\quad \Ab \cdot \nu =0 \text{ on } \partial \Omega\,.
\end{equation}
Here, and in the continuation, $\nu$ denotes a unit normal to $\partial\Omega$,
pointing into the domain $\Omega$. If~\eqref{eq:5a} is not satisfied for say $\Ab_0$ satisfying \eqref{eq:0}, we can construct 
\[
\Ab = \Ab_0 + \nabla \phi
\]
satisfying in addition \eqref{eq:5a}, by choosing $\phi$ as a solution of
\[
-\Delta \phi = \Div \Ab_0\text{ in } \Omega\,;\quad  \nabla \phi\cdot \nu = - \Ab_0\cdot \nu \text{ on } \partial \Omega\,,
\]
which is unique if we add the condition $\int_\Omega \phi (x)\, dx =0\,$.

\subsection{ The role of the circulations}\label{ss2.2}
The second point to observe is the following
\begin{proposition} Let $\Omega$ be an open connected set with $k$ holes 
$\Omega_j$, $j=1,\ldots,k$. Given $B$ in $C^\infty (\bar \Omega)$ and $k$ 
real numbers $\Phi_j$, $j=1,\ldots,k$,
then there exists a unique $\Ab$ satisfying \eqref{eq:0}, \eqref{eq:5a} and
\[
\int_{\partial \Omega_j} \Ab = \Phi_j,\quad j=1,\ldots,k.
\]
\end{proposition}
The proof is a consequence of the two lemmas below, giving separately 
uniqueness and existence.
\begin{lemma}\label{Lemma2.2}
If $\Ab$ and $\tilde \Ab$ satisfy \eqref{eq:0}  and \eqref{eq:5a}  with same circulations along the boundaries $\partial \Omega_j$, $j=1,\cdots,k$, then $\Ab=\tilde \Ab\,$.
\end{lemma}
\begin{proof}
If $\Ab$ and $\tilde \Ab$ correspond to the same $B$ and have same circulations 
along $\partial \Omega_j$, $j=1,\ldots,k$, they differ by a gradient $\nabla \phi$:
\[
\Ab - \tilde \Ab= \nabla \phi\,.
\] 
This gradient should satisfy $\Div \nabla \phi= \Delta \phi=0$ and 
$\partial_\nu \phi =0$ on $\partial \Omega$. Hence $\phi$ 
should be constant in $\Omega$ (we have indeed assumed that $\Omega$ is connected). Hence $\nabla \phi=0$ and the lemma follows.
\end{proof}

The existence is obtained through

\begin{lemma}\label{lemma2.3}
Given $B$ in $C^\infty (\overline{ \Omega})$ and $k$ real numbers $\Phi_j$, $j=1\ldots,k$, then there exists $\Ab$ satisfying \eqref{eq:0}, \eqref{eq:5a} and
\[
\int_{\partial \Omega_j} \Ab = \Phi_j,\quad j=1,\ldots,k.
\]
\end{lemma}
\begin{proof}
We first extend $B$ to $\widetilde B$ in $C^\infty (\overline{\widetilde \Omega})$  in an arbitrary way. Here $\widetilde \Omega$ is the simply connected hull of~$\Omega$:
\[
\widetilde \Omega = \Omega \cup \Bigl(\bigcup_{j=1}^k \overline {\Omega}_j\Bigr) \,.
\]
Next, we add to this extension $\widetilde B$ a smooth magnetic field $\sum_{j=1}^k \beta_j$ such that
$\supp \beta_j \subset \Omega_j$ for $j=1,\ldots, k$ and such that  the flux of $\widehat B = \widetilde B + \sum \beta_j$ in $\Omega_j$ is $ \Phi_j$. Let $\widehat \Ab$ be some magnetic potential in $\widetilde \Omega$ associated with $\widehat B$. Considering the restriction of $\widehat \Ab$ to $\Omega$ and adding some gradient $\nabla \theta$  (as explained before) to have~\eqref{eq:5a} satisfied, we get the desired $\Ab$.
\end{proof}

\subsection{Canonical generating function of $\Ab$.}

In order to analyze the non-simply connected case, we first rediscuss the previous point and then recall the main points of the analysis in~\cite{EKP,HPS}. 
 \begin{proposition}\label{propexist}
 Let $\Omega$ be an open connected set with $k$ holes $\Omega_j$, $j=1,\ldots,k$. 
 If $\Ab=(A_1,A_2)$ satisfies  \eqref{eq:0} and \eqref{eq:5a}
   then there exists a unique $ \psi=\psi^\Ab $ (which will be called the 
canonical generating function of $\Ab$) such that 
\[
 \psi_{/\partial \widetilde\Omega} =0\,,
\]
 and 
 \begin{equation}\label{eq:5}
A_1 = - \partial_{x_2}  \psi\,,\, A_2 = \partial_{x_1}   \psi\,.
\end{equation}
Here $\widetilde\Omega$ denotes the simply-connected envelope of $\Omega$ 
and $\partial \widetilde\Omega$ its boundary.
 \end{proposition}
\begin{proof}
There is no problem if $\Omega$ is simply connected. It is indeed sufficient to 
solve~\eqref{eq:5a}. For $\Omega$ not being simply connected, the statement is 
maybe less standard (this is a particular case of the Hodge-De Rham theory in the case of a manifold with boundary).  We observe that if $\Div \Ab =0$, then
the $1$-form  $\omega:=  -A_2 dx_1 + A_1 dx_2$ 
is closed (hence locally exact). In order to verify the existence of a
global $ \psi$ such that $d  \psi =\omega$, we have just to make 
sure that the integral of $ \omega$ along
each closed path is zero. This is reduced to the verification  that 
$ \int_{\partial \Omega_j} \omega =0$, which is a consequence  of the property 
that $\Ab\cdot \nu =0$ on $\partial \Omega\,$.  The function   $\psi$ is constant on each component of $\partial \Omega$. 
The uniqueness is obtained by 
imposing $ \psi_{/\partial \widetilde \Omega} =0$.
\end{proof}

We fix some notation. If $\Ab=(A_1,A_2)$ satisfies  \eqref{eq:0} and \eqref{eq:5a}, $\psi^\Ab$ depends only on $B$ and (in the case with $k$ holes) on the circulations $\Phi=(\Phi_1,\cdots,\Phi_k)$. We prefer in the future write $\psi^\Phi$ instead of $\psi^\Ab$. We  write also, for $\psi:=\psi^\Phi$, (or $\psi=\psi_0$ in the simply 
connected case)
\begin{equation}\label{defpsimin}
\psi_{min}:=\inf \psi,\quad
\psi_{max} = \sup \psi,\quad
\Osc ( \psi) := \psi_{max} - \psi_{min}.
\end{equation}
\begin{remark}
Note that by the maximum principle, under assumption that $B>0$ and $\Omega$ is simply connected, we have 
$\psi_0 \leq  0$,  $\inf \psi_0 <0$, and there 
exists at least a point $x_0 \in \overline{  \Omega} $ such that $\psi (x_0)= \inf \psi$. Hence we have
\[
\Osc ( \psi_0) = - \psi_{min} \text{ if } \Omega \text{ simply connected and } B>0\,.
\]
We stress that in the non-simply connected situation $x_0$ could belong to $\partial \Omega$.
\end{remark}

\section{Proof of Theorem \ref{thmEKPNSC}}\label{s3}

To get lower bounds, we consider the quadratic form and make the substitution
\[
u = \exp \Bigl(- \frac{\psi ^\Phi}{h}\Bigr)\,  v\,.
\] 
We have the following identity (see (2.4) in \cite{EKP}) if $\Ab$ satisfies \eqref{eq:0} and  \eqref{eq:5a}:
\begin{equation}\label{eq:pr1}
\| (hD-\Ab) u\|^2 -h \int_\Omega B(x) |u(x)|^2\, dx  
= h^2 \int_\Omega \exp \Bigl(- 2 \frac{\psi^\Phi}{h}\Bigr) |(\partial_{x_1} + i \partial_{x_2}) v|^2\, dx\,.
\end{equation}
We assume now  that $u$ (and consequently  $v$) is in $H_0^1(\Omega)$ and
estimate the quadratic form from below,

\begin{equation*}
\begin{aligned}
\| (hD-\Ab) u\|^2 -h \int_\Omega B(x) |u(x)|^2\, dx 
& \geq h^2 \exp \Bigl(\frac{- 2 \psi_{max}}{h}\Bigr)  \int_\Omega |(\partial_{x_1} + i \partial_{x_2}) v|^2\, dx \\
& \geq h^2 \exp \Bigl(\frac{- 2 \psi_{max}}{h}\Bigr)  \int_\Omega |\nabla  v|^2\, dx\\
& \geq h^2 \exp \Bigl(\frac{- 2 \Osc (\psi^\Phi)}{h}\Bigr)  \lambda^D(\Omega) \int_\Omega |u|^2 dx.
\end{aligned}
\end{equation*}
In terms of the lowest eigenvalue of $P_-$, this inequality for the quadratic 
form implies
\begin{theorem}\label{th2.1}
Assume that $\Omega$ is a bounded connected domain. Then,
for $h>0$,
\[
 \lambda_{P_-}^D(h,\Ab, B,\Omega) \geq h^2  \lambda^D(\Omega) \, \exp \Bigl(\frac{ -2 \Osc(\psi^\Phi)}{h}\Bigr) \,,
\]
when $k>0$ and
\[
 \lambda_{P_-}^D(h,B,\Omega) \geq h^2  \lambda^D(\Omega) \, \exp \Bigl(\frac{ -2 \Osc(\psi_0)}{h}\Bigr) \,,
\]
when $k=0$.
Here, as  before, $\lambda^D(\Omega)$ denotes the ground state energy of the
Dirichlet Laplacian in $\Omega$.
\end{theorem}
This is nothing else than the statement in \cite{EKP} for a specific choice of $\psi=\psi^\Phi$ or $\psi=\psi_0$. This gives Theorem \ref{thmEKPNSC},  with $\epsilon = 2  \Osc (\psi^\Phi)$ or $\epsilon = 2  \Osc (\psi_0)$ in the simply connected case. We note that $\epsilon >0$ when $B$ is not identically $0$. We have shown in~\cite{HPS} that the rate of the 
exponential decay is accurate when $B>0$  ($\epsilon = - 2 \inf \psi_0$) in the case of a simply connected $\Omega$.

We can find, using the maximum principle (see Subsection 7.4 in \cite{HPS}) a  lower  bound for $\psi_{min}$ by using the results obtained in the positive constant magnetic field (see \cite{EKP},\cite{HPS})   for specific  open sets $\widetilde \Omega$ (for example the disk).

\section{ Gauge invariance and first application}  \label{s4}
\subsection{Isospectrality}\label{ss2.4}
In Subsection~\ref{ss2.2} we have discussed the case when the magnetic potential 
corresponds to the same magnetic field and has same circulation. In the non-simply connected case, it is important to have in mind  the following proposition (see for example Proposition~2.1.3 in \cite{FH})
\begin{proposition}\label{gauge}
Suppose that $\Omega \subset \R^2$  is bounded and connected, that  $\Ab
\in C^1(\overline{\Omega})$, $\widetilde \Ab
\in C^1(\overline{\Omega})$. 
If  $\Ab$  and $\widetilde \Ab$ satisfy the following two conditions 
\begin{align}\label{betanul}
 \curl\Ab = \curl\widetilde \Ab  \;,
\end{align}
and
\begin{align}
\label{quantiz}
 \frac{1}{2\pi h} \int_{\gamma} (\Ab -\widetilde \Ab)   \in \mathbb Z 
\end{align}
 on any closed path $\gamma$ in $\Omega$,
 then the associated Dirichlet realizations of the  Schr\"odinger magnetic operators $ (hD-\Ab)^2 + V$ and $(hD-\widetilde \Ab)^2 + V $ are unitary equivalent.
\end{proposition}
This can in particular be applied to the case of the Pauli operator  with $V=\pm h B$.

\subsection{First application}
The lower bound in Theorem \ref{th2.1} can be improved by observing that, by  Proposition \ref{gauge}, one can optimize $\Osc (\psi^\Phi)$  over the 
$\Phi_j$ modulo $2\pi h \mathbb Z $.  

\begin{theorem}\label{th2.1a}
Assume that $\Omega$ is a bounded connected domain, with $k$ holes ($k>0$).
For $h>0$,  
\[
 \lambda_{P_-}^D(h,\Ab, B,\Omega) \geq h^2  \lambda^D(\Omega) \, \exp \frac{- 2}{h} \left(\inf_{\alpha \in \mathbb Z^n}  \Osc(\psi^{\Phi + 2\pi \alpha h}) \right)  \,.
\]
\end{theorem}
Hence, it remains to analyze the quantity:
\[
\delta(\Phi,h) = \left(\inf_{\alpha \in \mathbb Z^n}  \Osc(\psi^{\Phi + 2\pi \alpha h}) \right)  - \Osc (\psi_0)\,.
\]

\section{On the links between the circulations and the restrictions of $\psi$ at the boundary}\label{s7}

\subsection{On the links}
To treat the non  simply connected case more concretely, we relate  the values $ p_i$  of $\psi$ on the different components $\partial \Omega_i$ (which we below assume oriented counterclockwise)  of the boundary with the circulations $\Phi_j$ in each hole $\Omega_i$. We will see that this problem, which is usually treated in the Hodge--de Rham-theory for problems with boundary,  can in our situation be reformulated
as the problem of invertibility of a certain linear map.

The starting point is to construct $k$ functions $\theta_j$ as the solutions
of
\[
\Delta \theta_j = 0 \,,\, (\theta_j)_{/\partial\Omega_j} =1\,,\, (\theta_j)_{/\partial\Omega_{i}}=0 \text{ for } i\neq j\,.
\]
Note that by the Maximum principle
\begin{equation}\label{conttheta}
0 \leq \theta_j \leq 1\,.
\end{equation}
We consider for $p\in \mathbb R^k$ a solution of 
\[
\Delta \psi_p (x) = B (x)\,,\, (\psi_p)_{/\partial\Omega_0} =0\,, \, (\psi_{p})_{/\partial\Omega_j} =p_j \text{ for } j=1,\ldots, k \,.
\]
Let us observe that
\begin{equation}\label{formulapsip}
\psi_p - \psi_q = \sum_j (p_j- q_j) \theta_j\,.
\end{equation}
Let $\Ab_p$ be the associated magnetic potential:
\[
\Ab_p:= (-\partial_{x_2}\psi_p,\partial_{x_1}\psi_p).
\]
The circulation of $\Ab_p$ along $\partial\Omega_j$ is given by 
\[
\Phi (\Omega_j,p):=  \int_{\partial\Omega_j} \Ab_p \,ds  
= \int_{\partial\Omega_j} ( \nu\cdot \nabla \psi_p) \,.
\]
Hence, we get an affine map $\mathbb R^k \ni p\mapsto \Phi (p) \in \mathbb R^k$\,, where for $i=1,\cdots,k$, 
\begin{equation}\label{Phip}
\Phi_i (p) :=\Phi (\Omega_i,p) =  \Phi (\Omega_i,0) +  \int_{\partial\Omega_i} (\Ab_p-\Ab_0) \,   ds =  \Phi (\Omega_i,0) +  \sum_{j} M_{ij} \,  p_j\,,
\end{equation}
where
\[
M_{ij}:= \left( \int_{\partial\Omega_i}  (\nu\cdot \nabla \theta_j)\, ds\right) \,,\, (i,j)\in \{1,\dots, k\}^2\,,
\]

\begin{lemma}
$M$ is invertible.
\end{lemma}
\begin{proof}
It suffices to show the injectivity. But that is a direct consequence of 
Lemma~\ref{Lemma2.2}.
\end{proof}

We denote by $\mathbb R^k \ni \Phi \mapsto p(\Phi)$ the inverse map and note that we have
\[
\begin{gathered}
{\psi}_p = {\psi}^\Phi \text{ for }  \Phi=\Phi (p) \text{ or equivalently } p= p(\Phi)\,,\\
{\Ab}_p = {\Ab}^\Phi \text{ for }  \Phi=\Phi (p) \text{ or equivalently } p= p(\Phi)\,,\\
\end{gathered}
\]
We also write $\Phi_0 = \Phi (0)$ and observe that $\psi^{\Phi_0} = \psi_0$.

\subsection{Variation of the oscillation}
The oscillation is defined by
\[
 \Osc ( \psi_p) = \sup \psi_p - \inf \psi_p\,.
\]
Using~\eqref{conttheta} and~\eqref{formulapsip} we get:
\begin{equation} \label{contosc}
| \Osc (\psi_p) - \Osc (\psi_0)| \leq \sum_j |p_j| \,.
\end{equation}
Similarly, using \eqref{Phip} and the invertibility of $M$, we get, for some constant $C>0$
\begin{equation}\label{contosc2}
| \Osc (\psi^\Phi) - \Osc (\psi_0)| \leq C \, |\Phi -\Phi_{0}|\,.
\end{equation}

\subsection{Implementing gauge invariance}\label{ss6.2}
We are look for an optimal upper bound of
\[
\inf_{\alpha \in \mathbb Z^k} \Osc \psi_{p ( \Phi + 2\pi  h \alpha )}\,.
\]
Using the control of the oscillation established in~\eqref{contosc2}, we obtain
\begin{theorem}
Assume that $ \Omega$ is a bounded connected domain with $k$ holes. 
Then there exists $ C>0$ such that, for any $\Phi\in \mathbb  R^k$, any $\ h>0$
\[
 \lambda_{P_-}^D(h,\Ab, B,\Omega)
 \geq h^2  \lambda^D(\Omega) \, \exp \Bigl(- \frac C h d(\Phi, \Phi_0 + 2\pi h \mathbb Z^k )
\Bigr)\,\exp \biggl(\frac{-2 \Osc  (\psi_0) }{h}\biggr) \,.
\]
When $ B>0$, we get
\[
 \lambda_{P_-}^D(h,\Ab, B,\Omega) \geq h^2  \lambda^D(\Omega) \, \exp \Bigl(- \frac C h d(\Phi, \Phi_0 + 2\pi h \mathbb Z^k )\Bigr)\,\exp \biggl(\frac{2  (\inf \psi_0) }{h}\biggr) \,.
\]
\end{theorem}

\begin{corollary} \label{th1.3nscbis}
 If $B(x) >0$,  $\Omega$ is connected, and if $\psi_0$ is the solution of 
\[
\Delta \psi_0 = B(x) \text{ in } \Omega \,,\, \psi_{0/\partial \Omega} =0\,,
\]
then  for any $\Ab$ such that $\curl \Ab=B$ the  groundstate energy of the associated  Pauli operator satisfies
\[
\liminf_{h\rightarrow 0} h \log  \lambda_{P_-}^D(h,\Ab,B, \Omega) \geq   2 \inf \psi_0\,.
\]
\end{corollary}

\begin{remark}\label{Rem5.4}
If $B >0$ in $\Omega$, the minimal oscillation of $\psi$ when $\psi$ is a solution of $\Delta \psi = B$ in $\Omega$ and $\psi=0$ on $\partial \widetilde \Omega$ is obtained when $\psi$ satisfies in addition  $\psi =0$
 on $\partial \Omega$.
 This is an immediate consequence of the maximum principle (see Subsection 7.4 in \cite{HPS}). Hence the oscillation is minimal for $\Phi=\Phi_0$ and the  corresponding $\Ab$.
 \end{remark}

\section{Upper bounds in the general case with $B>0$}\label{s6}

\subsection{Preliminary discussion}
In the simply connected case, with the explicit choice of $\psi$, it is easy to get:
 \begin{proposition}
If $B>0$  and assuming that $\Omega$ is simply connected,  we have, for any $\eta >0\,$,
\[
 \lambda_{P_-}^D(h,B,\Omega) \leq C  \exp \Bigl(\frac{2 \psi_{min}}{h}\Bigr)
\exp\Bigl(\frac {2\eta}{ h}\Bigr) \,.
\]
\end{proposition}
The proof is obtained by taking as trial state  $u= \exp \bigl(- \frac{\psi}{h}\bigr) v_\eta$, with $v_\eta $ with compact support in $\Omega$ and $v_\eta=1$ outside a sufficiently small neighborhood of the boundary and implementing this quasimode in \eqref{eq:pr1}. One concludes by the max-min principle. It has been shown in \cite{HPS} how to have a (probably) optimal upper bound by using as trial state 
$u= \exp \bigl(- \frac{\psi}{h}\bigr) - \exp\bigl( \frac \psi h\bigr)$.

In the non simply connected case the situation is much more complicated.

We can get a general result by considering a simply connected open set  
$\widehat \Omega$  in $\Omega$.  This proves that  
$ \lambda_{P_-}^D(h,\Ab, B,\Omega)$ is indeed
exponentially small, independently of the circulations along each component of the boundary as soon as $B$ is positive somewhere.
 
 \begin{proposition}\label{prop6.2}
Assume that $\Omega$ is a bounded and connected domain and that the minimum $\psi_{min}$
of $\psi$ is attained in an interior point, i.e. 
$\min_{\partial\Omega}\psi>\psi_{min}$. Then, for any $\eta >0$, there 
exists $C_\eta >0$ and $h_\eta$ such that, for $h\in (0,h_\eta)$, 
 \begin{equation} \label{upbd}
 \lambda_{P_-}^D(h,\Ab, B,\Omega) \leq C  \exp\biggl( \frac{2 (\psi_{min} -  \min_{\partial\Omega}\psi ) }{h}\biggr)  \exp \Bigl( \frac {2\eta}{ h}\Bigr) \,.
\end{equation}
 \end{proposition}

\begin{remark} We observe that the condition  $\min_{\partial\Omega}\psi > \psi_{min}$ is stable when a small variation of the circulations $\Phi_j$ is performed. 
\end{remark}

\subsection{Application}
The next question is to prove the upper-bound. We have already shown that for $p=0$ we can find $\Ab$ such that the corresponding circulations $\Phi=(\Phi_1,\cdots,\Phi_k)$ satisfy $p(\Phi)=0$ (with the notation of Section \ref{s7}). 
The proof of the upper-bound in the case $p=0$ is then the same as in the simply-connected case. We choose $\alpha (h) \in \mathbb Z^k$ such that 
\[
\bigl|\Phi + 2\pi h \alpha(h) - \Phi_0\bigr| \leq C h\,.
\]
For  $\Phi^{new}_h:= \Phi + 2\pi h \alpha(h)$,  $p(\Phi^{new}_h)$ is $\mathcal O (h)$, the solution $\psi_{p(\Phi^{new}_h)}$ is  $\mathcal O (h)$ close to $\psi_0$. Its infimum is realized inside $\Omega$. 
We can then use \eqref{upbd} for $\Ab:= \Ab^{\Phi^{new}_h}$ and get 
for any $\eta >0$, the existence of 
exists $C_\eta >0$ and $h_\eta$ such that, for $h\in (0,h_\eta)$, 
 \begin{equation} \label{upbd}
 \lambda_{P_-}^D(h,\Ab^{\Phi^{new}_h}, B,\Omega) \leq C  \exp\Bigl( \frac{2 (\inf \psi_{0}  + \mathcal O (h) ) }{h}\Bigr)  \exp \Bigl( \frac {2\eta}{ h}\Bigr) \,.
\end{equation}
Playing with $\eta >0$ and using that $ \lambda_{P_-}^D(h,\Ab^{\Phi^{new}_h}, B,\Omega)=   \lambda_{P_-}^D(h,\Ab^\Phi,B, \Omega)$, we finally obtain:
\begin{proposition} \label{th1.3nscter}
 If $B(x) >0$,  $\Omega$ is connected, and if $\psi_0$ is the solution of 
\[
\Delta \psi_0 = B(x) \text{ in } \Omega \,,\, \psi_{0/\partial \Omega} =0\,,
\]
then  for any $\Ab$ such that $\curl \Ab=B$ the  groundstate energy of the associated  Pauli operator satisfies
\[
\lim\sup_{h\rightarrow 0} h \log  \lambda_{P_-}^D(h,\Ab,B, \Omega) \leq   2 \inf \psi_0\,.
\]
\end{proposition}
\subsection{Conclusion}
With Corollary \ref{th1.3nscbis} and Proposition \ref{th1.3nscter} we have achieved the proof of Theorem \ref{th1.3nsc}.

\section{Analysis of an example -- the annulus}\label{sannulus}

\subsection{Introduction}\label{ss4.1} 
Throughout this section we consider $\Omega$ to be an annulus and the magnetic
field to be uniform and positive. We stress that this is a particular case covered
in the previous sections. Our aim is to give more precise estimates in this 
particular case and to make the connection between the trace $p$ and the
circulation $\Phi$ explicit. We consider the annulus
\[
A(\rho,R):=\{(x_1,x_2) \in \mathbb R^2\,,\, \rho ^2 < x_1^2+x_2^2 < R^2\}\,,
\]
and start with the  particular choice  $\Ab =(B/2)(-x_2, x_1)$ for some $B>0$. 
This corresponds to the case of a constant magnetic field of strength $B$.

In this case we have
\begin{equation*}\label{eq:19} 
\psi^\Ab (x_1,x_2) = \frac{B(x_1^2+ x_2^2-R^2)}{4}\quad
\text{and}\quad 
\psi^{\Ab}_{min}=\frac{B(\rho^2-R^2)}{ 4}\,.
\end{equation*}

Theorem \ref{th2.1} gives the following lower bound for general constant 
magnetic field~$B$ and $h>0$:
\begin{theorem}\label{th4.1}
The ground state energy of the Pauli operator with  $\Ab  = (B/2) (-x_2, x_1)$ 
and $h>0$ satisfies
\[
\lambda^D_{P_-} (h,\Ab,B,A(\rho,R)) \geq h^2\, \frac{ {\bf j^2} }{\pi (R^2-\rho^2)} \exp \bigl(- (R^2-\rho^2) B/ (2h) \bigr)\,,
\]
where $\bf j$ is the smallest positive zero of the Bessel function $J_0$.
\end{theorem}
This corresponds to the model analyzed in~\cite{EKP} in the non-simply 
connected case. The case of the disk was first considered in~\cite{E},~\cite{EKP} 
and~\cite{HPS} (see also references therein). Note that this lower bound is 
universal but the general theory of the previous sections shows   that, except 
for the disk,  we are far from from optimality. We have indeed 
not used the improvment using the gauge invariance.

\subsection{General circulation}\label{ss4.2}
A radial solution  of $\Delta \psi =B$ has necessarily the form
 \[
 \psi (r) = B \frac {r^2}{4} + C \log r  + D,\quad r=|x|.
 \]
By a scaling, one can normalize (at the prize of changing $B$) by saying that the exterior radius is 
$R= 1$ and that $\psi(1)=0$.
Observing that if $\psi$  is a solution of $\Delta \psi =1$, then 
$\psi_B:= B \psi$ satisfies $\Delta \psi_B=B$, we can reduce our analysis to 
$R=1$ and $B=1$ and this is the case which is considered below.
Hence the general solution such that $\psi(1)=0$  is given by
\[
\widehat \psi^C(r) = \frac{r^2}{4} - \frac 14 + C \log r
\]
where the constant  $C$ has to be related  with the circulation  $\Phi =\Phi_1 $ along  the interior circle of radius $\rho$ or with the trace $p_1=p$ of $\widehat \psi^C$ on the same circle:
\begin{equation}
p:= \frac{\rho ^2}{4} - \frac 14 +  C \log \rho\,.
\end{equation}
Furthermore, since $\Delta\log r=2\pi\delta_0$, where $\delta_0$ denotes the
Dirac measure at the origin, we have
 \begin{equation}\label{formFlux}
\Phi = 2\pi \, C + \pi \rho^2\,,
\end{equation}
and with the notation of the previous sections,
\begin{equation}
\psi^\Phi= \psi_p = \widehat \psi^C\,.
\end{equation}
In our particular case, the link between $\Phi =\Phi_1$ and $p=p_1$ the value of $\psi$ on the interior circle $\partial \Omega_1$ reads
\begin{equation}
\frac{\Phi}{2 \pi}= \frac{1}{\log \rho} p    +  \rho^2  + \frac{1-\rho^2}{4 \log \rho}\,.
\end{equation}

\subsection{Minimization over $C$ of the oscillation of $\widehat \psi^C$.}
We next analyze the variation of $\psi^C$. We have:
\[
(\widehat\psi^C)'(r) = \frac r 2 +  \frac C  r\,.
\]
Hence, the only possible critical points must satisfy $r^2 = -  2 C$. We get
three different cases.

For $C \geq -\rho^2/2$, there is no minimum inside $(\rho, 1)$. Hence
\[
\widehat\psi_{min}^C = \frac{\rho^2}{ 4} -\frac 14+ C \log \rho <0,\quad \widehat\psi^C_{max}=0\,,
\]
the maximum being attained at the inner radius, that is for $r=\rho$. Hence, in
this case, the oscillation is
\begin{equation}\label{OscpsiC1}
\Osc (\widehat\psi^C):= \widehat\psi_{max}^C-\widehat\psi_{min}^C=  - \frac{\rho^2}{ 4} + \frac 14 -   C \log \rho. 
\end{equation}

For $C <  -1/2$,  there are also no minimum inside $(\rho, 1)$. 
The minimum is  for $r=1$ and the maximum is for $r=\rho$. One has 
 \[
\widehat\psi^C_{max} =   \frac{\rho^2}{ 4} -\frac 14+ C \log \rho,\quad \psi^C_{min}=0\,.
 \]
Hence the oscillation is 
\begin{equation}\label{OscpsiC2}
\Osc (\widehat\psi^C) =  \frac{\rho^2}{ 4} - \frac 14 +  C \log \rho.
\end{equation}

Finally, for $C \in (-1/2, -\rho^2/2)$, we have the maximum at 
the boundary
\[
\widehat\psi_{max}^C = \max \Bigl(0, \frac{\rho^2}{ 4} -\frac 14+ C \log \rho \Bigr) \,,
\]
a minimum in $(\rho,1)$  such that 
 \[
\widehat \psi_{min}^C = - \frac C 2  - \frac 14 + \frac C2  \log (-2C)\,.
 \]
Thus 
 \begin{equation}\label{OscpsiC3}
\Osc (\widehat\psi^C)=  \max \Bigl(0, \frac{\rho^2}{ 4} -\frac 14 
+ C \log \rho \Bigr)   + \frac C 2  + \frac 14 - \frac C2  \log (-2C).
\end{equation}
 
By direct computation, one can recover the general result  (see 
Remark~\ref{Rem5.4}) that the oscillation of $\widehat\psi^C$ is minimal for 
$C$ such that $\widehat\psi^C (\rho)=0$.

\begin{proposition}\label{propminoscannulus}
The infimum over $C$ of $\Osc (\widehat \psi^C)$  is 
\[
\frac {C_{crit}}{ 2}  + \frac 14 - \frac {C_{crit}}{2}  \log (-2C_{crit}) >0\,,
\]
with $C_{crit}$ defined by
\begin{equation}\label{defCcrit}
C_{crit}:=    \frac{1- \rho^2}{4  \log \rho}\,.
\end{equation}
\end{proposition}

\begin{figure}[h]
\centering
\includegraphics{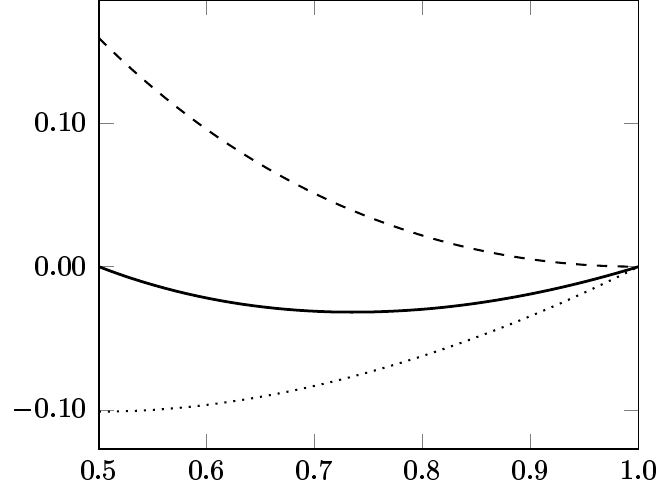}
\caption{Graph of $\widehat \psi^C$ with $\rho=0.5$ in the cases 
 $C=-1/2$ (dashed), $C=C_{crit}\approx -0.27$ (solid), and $C=-\rho^2/2$ (dotted).}
\label{fig:psic}
\end{figure}

\begin{remark}
The analysis given in this section can be extended to the case of a non-uniform
radial magnetic field.
\end{remark}

\subsection{Numerical simulations for the annulus}

We close this section by considering, numerically, the bottom of the spectrum
when $\rho=1/2$. Thus, we consider
\[
A(1/2,1)=\{(x,y)\in\mathbb R^2,\ (1/2)^2<x_1^2+x_2^2<1\}, 
\]
and use the magnetic potential
\[
\Ab=(A_1,A_2)=\frac{1}{2}(-x_2,x_1)+\frac{\kappa}{x_1^2+x_2^2}(-x_2,x_1).
\]
Here, the parameter $\kappa$ denotes the flux of an Aharonov--Bohm solenoid 
located at the origin, $\kappa=\Phi/2\pi$. Thus $\kappa$ is directly linked to 
$C$ and $p$. We investigate how the lowest eigenvalue of the Pauli
operator corresponding to this magnetic potential depends on $\kappa$ (for small
$h$!), and to link it with our previous discussion. The Pauli operator
\[
P=(hD_{x_1}-A_1)^2+(hD_{x_2}-A_2)^2-h
\]
can, using polar coordinates $x_1=r\cos\theta$, $x_2=r\sin\theta$, 
written as
\[
-h^2\frac{\partial^2}{\partial r^2} - h^2\frac{1}{r}\frac{\partial}{\partial r}+
\Bigl(\frac{r}{2}-\frac{h(-i\partial_\theta)-\kappa}{r}\Bigr)^2-h
\]
With the usual angular momentum decomposition we are led to the family of
self-adjoint ordinary differential operators
\[  
P_m=-h^2\frac{d^2}{d r^2} - h^2\frac{1}{r}\frac{d}{d r}+
\Bigl(\frac{r}{2}-\frac{hm-\kappa}{r}\Bigr)^2 -h \,,\quad m\in\mathbb Z,
\]
each with Dirichlet boundary conditions.
We study the situation for $h=0.1$, $h=0,01$ and $h=0.001$.
The spectrum of $P$ is given as the union of the spectrum of the operators $P_m$.
In particular, the lowest point of the spectrum of $P$ is given by the minimum of the
first eigenvalues $\lambda_m(h,\kappa)$ of the operators $P_m$.

We discretize the eigenvalue problem and solve the discretized problem with 
an iterative method (using the \textsc{Scipy} library for \textsc{Python}).
The results can be seen in Figure~\ref{fig:h01} for $h=0.1$, in Figure~\ref{fig:h001} for $h=0.01$ and  in Figure~\ref{fig:h0001} for $h=0.001$.

\begin{figure}[h]
\centering
\includegraphics{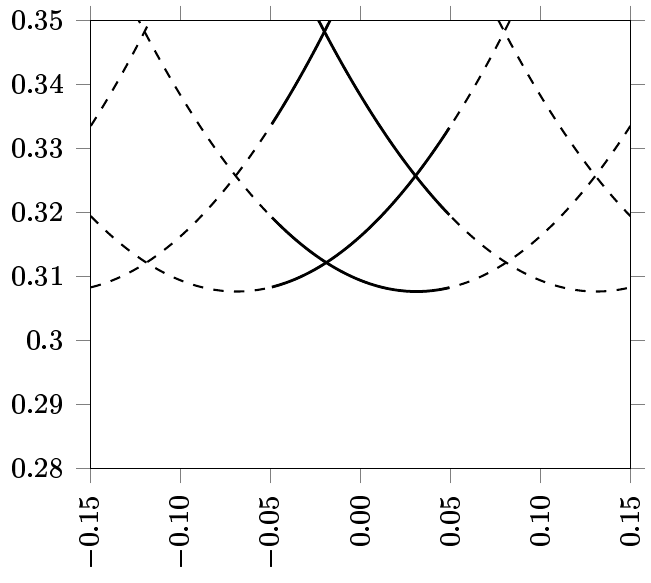}
\caption{Here we see the graphs of $\kappa\mapsto \lambda_m(h,\kappa)$, 
$-1.5\,h\leq \kappa\leq 1.5\,h\,$, $h=0.1$ and $0\leq m\leq 5\,$. Since the periodicity is $h\,$, we
have marked one period in bold.}
\label{fig:h01}
\end{figure}

\begin{figure}[h]
\centering
\includegraphics{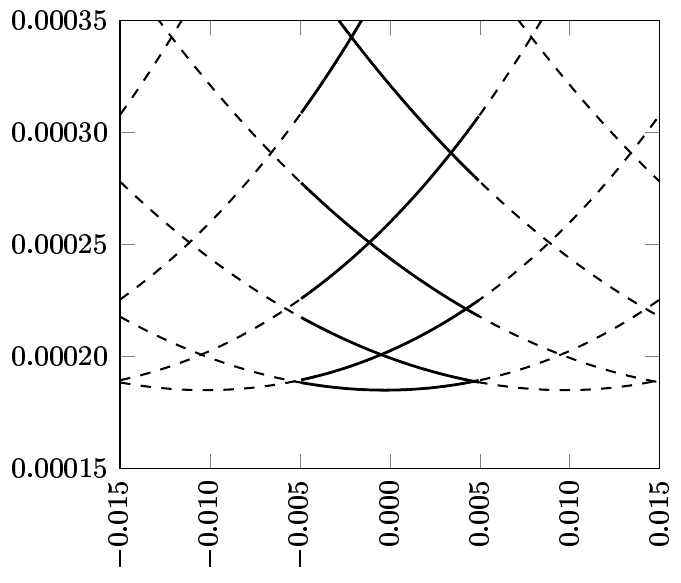}
\caption{Here we see the graphs of $\kappa\mapsto \lambda_m(h,\kappa)$, 
$-1.5\,h\leq \kappa\leq 1.5\,h\,$, $h=0.01$ and $23\leq m\leq 31\,$. Since the periodicity is $h\,$, we
have marked one period in bold.}
\label{fig:h001}
\end{figure}

\begin{figure}[h]
\centering
\includegraphics{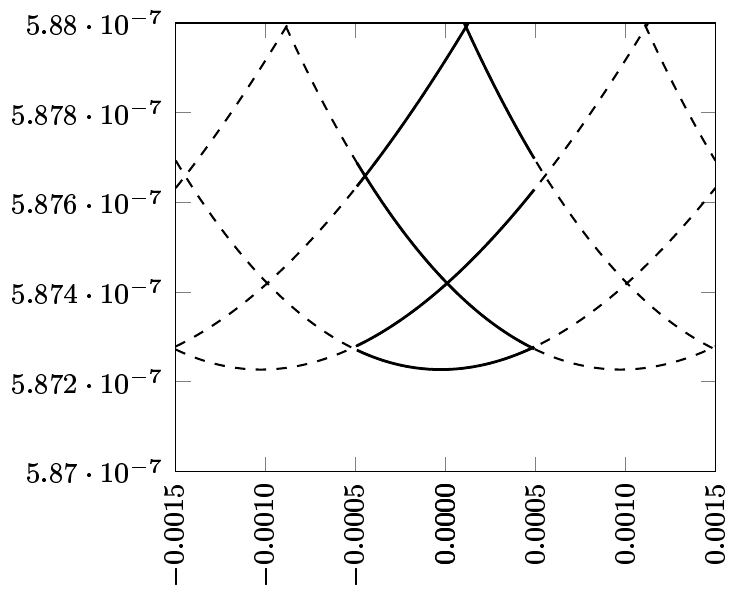}
\caption{Here we see the graphs of $\kappa\mapsto \lambda_m(h,\kappa)$, 
$-1.5\,h\leq \kappa\leq 1.5\,h\,$, $h=0.001$ and $386\leq m\leq 392$. Since the periodicity is $h\,$, we
have marked one period in bold.}
\label{fig:h0001}
\end{figure}

\begin{remark}
The eigenvalue equation $P_mu=\lambda u$ can in principle be solved in terms
of Whittaker functions. Imposing the boundary conditions one get an equation
in $\kappa$ and $\lambda$, that can be solved numerically. We tried this
approach, using, Wolfram Mathematica, but it turned out that we hit some 
exceptional values for the Whittaker functions, giving spurious extra solutions.
\end{remark}

\section*{Acknowledgements} 
During this work, the first author  was partially supported by the ANR Nosevol 
and the University of Lund. The first author thanks also Nicolas Raymond for an helpful discussion on Hodge-De Rham theory.

\FloatBarrier

\bibliographystyle{plain}

\end{document}